\documentclass[10pt,reqno]{amsart}
\usepackage{latexsym,amsmath,amssymb,amscd}
\usepackage[all]{xy}
\usepackage{enumerate}
\usepackage{array}
\usepackage{hyperref}

\makeatletter
\newcommand\@dotsep{4.5}
\def\@tocline#1#2#3#4#5#6#7{\relax
	\ifnum #1>\c@tocdepth 
	\else
	\par \addpenalty\@secpenalty\addvspace{#2}%
	\begingroup \hyphenpenalty\@M
	\@ifempty{#4}{%
		\@tempdima\csname r@tocindent\number#1\endcsname\relax
	}{%
		\@tempdima#4\relax
	}%
	\parindent\z@ \leftskip#3\relax \advance\leftskip\@tempdima\relax
	\rightskip\@pnumwidth plus1em \parfillskip-\@pnumwidth
	#5\leavevmode\hskip-\@tempdima #6\relax
	\leaders\hbox{$\m@th
		\mkern \@dotsep mu\hbox{.}\mkern \@dotsep mu$}\hfill
	\hbox to\@pnumwidth{\@tocpagenum{#7}}\par
	\nobreak
	\endgroup
	\fi}
\makeatother 

\begin{document}
	
	
	\makeatletter
	\@addtoreset{figure}{section}
	\def\thefigure{\thesection.\@arabic\c@figure}
	\def\fps@figure{h,t}
	\@addtoreset{table}{bsection}
	
	\def\thetable{\thesection.\@arabic\c@table}
	\def\fps@table{h, t}
	\@addtoreset{equation}{section}
	\def\theequation{
		\arabic{equation}}
	\makeatother
	
	\newcommand{\bfi}{\bfseries\itshape}

	\newtheorem{theorem}{Theorem}[section]
	\newtheorem{corollary}[theorem]{Corollary}
	\newtheorem{definition}[theorem]{Definition}
	\newtheorem{example}[theorem]{Example}
	\newtheorem{examples}[theorem]{Examples}
	\newtheorem{lemma}[theorem]{Lemma}
	\newtheorem{notation}[theorem]{Notation}
	\newtheorem{problem}[theorem]{Problem}
	\newtheorem{proposition}[theorem]{Proposition}
	\newtheorem{remark}[theorem]{Remark}
	\numberwithin{equation}{section}

	\renewcommand{\1}{{\bf 1}}
	\newcommand{\Ad}{{\rm Ad}}
	\newcommand{\Aut}{{\rm Aut}\,}
	\newcommand{\ad}{{\rm ad}}
	\newcommand{\card}{{\rm card}}
	\newcommand{\Ci}{{\mathcal C}^\infty}
	\newcommand{\Der}{{\rm Der}}
	\newcommand{\de}{{\rm d}}
	\newcommand{\ee}{{\rm e}}
	\newcommand{\End}{{\rm End}\,}
	\newcommand{\ev}{{\rm ev}}
	\newcommand{\GL}{{\rm GL}}
	\newcommand{\Gr}{{\rm Gr}}
	\newcommand{\graf}{{\mathsf{G}}}
	\newcommand{\Hom}{{\rm Hom}}
	\newcommand{\hotimes}{\widehat{\otimes}}
	\newcommand{\id}{{\rm id}}
	\newcommand{\ie}{{\rm i}}
	\newcommand{\gl}{{{\mathfrak g}{\mathfrak l}}}
	\newcommand{\Ker}{{\rm Ker}\,}
	\newcommand{\LD}{\pounds}
	\newcommand{\Lie}{\text{\bf L}}
	\newcommand{\pr}{{\rm pr}}
	\newcommand{\Ran}{{\rm Ran}\,}
	\renewcommand{\Re}{{\rm Re}\,}
	\newcommand{\spa}{{\rm span}\,}
	\newcommand{\Tr}{{\rm Tr}\,}
	\newcommand{\U}{{\rm U}}
	
	\newcommand{\CC}{{\mathbb C}}
	\newcommand{\HH}{{\mathbb H}}
	\newcommand{\RR}{{\mathbb R}}
	\newcommand{\TT}{{\mathbb T}}
	
	\newcommand{\Ac}{{\mathcal A}}
	\newcommand{\Bc}{{\mathcal B}}
	\newcommand{\Cc}{{\mathcal C}}
	\newcommand{\Dc}{{\mathcal D}}
	\newcommand{\Ec}{{\mathcal E}}
	\newcommand{\Fc}{{\mathcal F}}
	\newcommand{\Gc}{{\mathcal G}}
	\newcommand{\Hc}{{\mathcal H}}
	\newcommand{\Jc}{{\mathcal J}}
	\newcommand{\Kc}{{\mathcal K}}
	\renewcommand{\Mc}{{\mathcal M}}
	\newcommand{\Nc}{{\mathcal N}}
	\newcommand{\Oc}{{\mathcal O}}
	\newcommand{\Pc}{{\mathcal P}}
	\newcommand{\Rc}{{\mathcal R}}
	\newcommand{\Sc}{{\mathcal S}}
	\newcommand{\Tc}{{\mathcal T}}
	\newcommand{\Vc}{{\mathcal V}}
	\newcommand{\Uc}{{\mathcal U}}
	\newcommand{\Xc}{{\mathcal X}}
	\newcommand{\Yc}{{\mathcal Y}}
	\newcommand{\Zc}{{\mathcal Z}}
	\newcommand{\Wc}{{\mathcal W}}

	\newcommand{\Ag}{{\mathfrak A}}
	\newcommand{\Bg}{{\mathfrak B}}
	\newcommand{\Fg}{{\mathfrak F}}
	\newcommand{\Gg}{{\mathfrak G}}
	\newcommand{\Ig}{{\mathfrak I}}
	\newcommand{\Jg}{{\mathfrak J}}
	\newcommand{\Lg}{{\mathfrak L}}
	\newcommand{\Mg}{{\mathfrak M}}
	\newcommand{\Pg}{{\mathfrak P}}
	\newcommand{\Sg}{{\mathfrak S}}
	\newcommand{\Xg}{{\mathfrak X}}
	\newcommand{\Yg}{{\mathfrak Y}}
	\newcommand{\Zg}{{\mathfrak Z}}
	
	\newcommand{\ag}{{\mathfrak a}}
	\newcommand{\bg}{{\mathfrak b}}
	\newcommand{\dg}{{\mathfrak d}}
	\renewcommand{\gg}{{\mathfrak g}}
	\newcommand{\hg}{{\mathfrak h}}
	\newcommand{\kg}{{\mathfrak k}}
	\newcommand{\mg}{{\mathfrak m}}
	\newcommand{\n}{{\mathfrak n}}
	\newcommand{\og}{{\mathfrak o}}
	\newcommand{\pg}{{\mathfrak p}}
	\newcommand{\sg}{{\mathfrak s}}
	\newcommand{\tg}{{\mathfrak t}}
	\newcommand{\ug}{{\mathfrak u}}
	\newcommand{\zg}{{\mathfrak z}}
	
	\newcommand{\bs}{\mathbf{s}}
	\newcommand{\bt}{\mathbf{t}}
	
	\newcommand{\ZZ}{\mathbb Z}
	\newcommand{\NN}{\mathbb N}
	\newcommand{\BB}{\mathbb B}
	\newcommand{\EE}{\mathbb E}
	\newcommand{\FF}{\mathbb F}
	\newcommand{\KK}{\mathbb K}
	\newcommand{\QQ}{\mathbb Q}

	\newcommand{\tto}{\rightrightarrows}
	
\newcommand{\todo}[1]{\vspace{5 mm}\par \noindent
\framebox{\begin{minipage}[c]{0.95 \textwidth}
\tt #1 \end{minipage}}\vspace{5 mm}\par}

	\title{Cyclic Lie-Rinehart algebras} 
	\author[D. Belti\c t\u a, A. Dobrogowska, G. Jakimowicz]{Daniel Belti\c t\u a, Alina Dobrogowska, Grzegorz Jakimowicz} 
	
	\address{Institute of Mathematics ``Simion Stoilow'' of the Romanian Academy, P.O. BoS 1-764, Bucharest, Romania}
	\email{Daniel.Beltita@imar.ro, beltita@gmail.com}

\address{Faculty of Mathematics, University of Bia\l{}ystok, Cio\l{}kowskiego 1M, 15-245 Białystok, Poland}
\email{alina.dobrogowska@uwb.edu.pl} 

\address{Faculty of Mathematics, University of Bia\l{}ystok, Cio\l{}kowskiego 1M, 15-245 Białystok, Poland}
\email{g.jakimowicz@uwb.edu.pl}

	
	

	
	\begin{abstract} 
We study Lie-Rinehart algebra structures in the framework provided by a duality pairing of modules over a unital commutative associative algebra. 
Thus, we construct  examples of Lie brackets corresponding to a fixed anchor map whose image is a cyclic submodule of the derivation module, and therefore we call them cyclic Lie-Rinehart algebras. 
In a very special case of our construction, these brackets turn out to be related to certain differential operators that occur in mathematical physics. 
	\end{abstract}
	
	\maketitle
	
	
	\section{Introduction}
	
	The notion of Lie-Rinehart algebra is the algebraic counterpart of the differential geometric notion of Lie algebroid, which in turn is a generalization of Lie algebras. 
	While the classification of Lie brackets on low-dimensional vector spaces 
	is a classical topic, the analogous problem for Lie-Rinehart algebras is much more difficult. 
	Yet, their deformation theory as well as classification in low dimensions have been recently discussed, e.g., in \cite{Re20} and \cite{BCEM20}. 
	See also \cite[\S 8]{KeWa15} for  earlier results in deformation theory of Lie-Rinehart algebras.
	
	On the other hand, the Lie-Rinehart algebras associated to Lie algebroids are usually infinite-dimensional vector spaces, hence the simplest ones should be small in a different sense than that of the dimension. 
	Starting from that idea, in the present paper we focus on the class of Lie-Rinehart algebras whose rank of the anchor map is equal to 1, 
	that is, the ones corresponding to a fixed anchor map whose image is a cyclic submodule of the derivation module, and therefore we call them \emph{cyclic Lie-Rinehart algebras}. 
	
	Our new constructions of cyclic Lie-Rinehart algebras are performed within the class of pre-Lie-Rinehart algebras, which in turn are a special case of Lie-admissible algebras, 
	cf. e.g., \cite{Re21} and the references therein. 
	See also \cite{LShBCh16} for the differential geometric version of this notion, 
	namely the left-symmetric algebroid. 
	In order to explain the initial motivation of our constructions, 
	we recall from \cite{Ri63}  some purely algebraic constructions associated to the $\KK$-algebra $R$, which are suggested by basic differential geometry on a smooth manifold~$M$, as summmarized in the following table: 
\begin{center}
	\begin{tabular}{ |m{6cm}  | m{2cm}| m{3cm} | } 
		\hline
		ground field	&	$\RR$ & $\KK$ \\ 
		\hline
		algebra of smooth functions	&	$\Ci(M,\RR)$ & $R$ \\ 
		\hline
		Lie algebra of tangent vector fields	&	$\Gamma(TM)$ &  $ \Der(R)$ \\ 
		\hline
		space of differential 1-forms	&	$\Omega^1(M)$ & $\Hom_R( \Der(R),R)$\\ 
		\hline
	\end{tabular}
\end{center}	
Here and throughout the present paper, $\KK$ is a field of characteristic zero and $R$ is a unital commutative associative $\KK$-algebra. 
In the above table, only the space of vector fields has the natural structure of a Lie algebra, while this is not the case for the spaces of smooth functions and differential 1-forms. 
However, these last two spaces can be endowed with Lie brackets using 
additional geometric data such as the Poisson structures or the Dirac structures. 
(See e.g., \cite{Do93}.) 
Starting from this remark, it is natural to ask what other Lie brackets are carried by the spaces of differential 1-forms. 
This question was addressed in a differential geometric setting in several recent works, including \cite{DJ21}, 
\cite{Ba21}, \cite{DJ23}, \cite{DJS22}. 
It is the aim of the present paper to address that question in a purely algebraic framework, using the method of pre-Lie-Rinehart algebras introduced in \cite{FMaMu21}.

{The structure of the present paper is as follows. 
In Section~\ref{Sect2} we collect preliminary remarks and we state the main problem towards whose solution we take a first step (Problem~\ref{basic}). 
In Section~\ref{Sect3} we establish some auxiliary facts on covariant differential operators. 
Along the way, we prove the transitivity of the Lie-Rinehart algebra associated to a projective module (Proposition~\ref{transit}), which bears on a problem pointed out in \cite[Rem. 4.19]{BKS22}. 
Section~\ref{Sect4} contains 
our main results, which provide new constructions of cyclic (pre-)Lie-Rinehart algebras.   See Theorems \ref{LSA_prop} and \ref{LSA_prop_new}, as well as their corollaries. 
Finally, in Example~\ref{A_ex} we illustrate the abstract results by some examples related to ordinary differential equations.

\section{The basic problem}
\label{Sect2}

In this preliminary section we fix notation and we state a classification problem on Lie-Rinehart structures with a fixed anchor map (Problem~\ref{basic}). 
The main results of the present paper are related to that problem in the special case of anchor maps whose image is a cyclic module (Definition~\ref{cyc_def}). 
	
	\begin{definition}[{\cite{Ri63}}]
		\normalfont 
		We define the following $R$-modules
		\begin{align*}
			T_R&:=\Der(R)=\{X\colon R\to R\mid X\text{ is a $\KK$-linear derivation}\}, \\
			\Omega^1_R&:=\Hom_R(T_R,R)=\{\alpha\colon T_R\to R\mid\alpha\text{ is $R$-linear}\}, \\
			\Omega^2_R&:=\{\omega\colon T_R\times T_R\to R\mid \omega\text{ is skew-symmetric $R$-bilinear}\},
		\end{align*}
		and the mappings
		$$\de\colon R\to \Omega^1_R,\quad  (\de r)(X):=X(r)\text{ for all }r\in R,\ X\in T_R$$ 
		and 
		$$\de\colon\Omega_R^1\to\Omega_R^2,\quad  (\de\alpha)(X,Y):=X(\alpha(Y))-Y(\alpha(X))-\alpha([X,Y]).$$ 
		Moreover, for every $X\in T_R$ we define 
		\begin{align}
			\label{LD1}
			\LD_X\colon\Omega_R^1\to\Omega_R^1,\quad 
			(\LD_X\alpha)(Y)
			:= & X(\alpha(Y))-\alpha([X,Y]) \\
			\label{LD2}
			=&\de\alpha(X,Y)+Y(\alpha(X))
			.
		\end{align}
	\end{definition}
	
	\begin{remark}
		\label{d_prop}	
		\normalfont
		For all $r,s\in R$ and $X\in T_R$ we have 
		$(\de(rs))(X)=X(rs)=X(r)s+rX(s)=(\de r\cdot s)(X)+(r\cdot\de s)(X)$ hence 
		\begin{equation}
			\label{d_prop_eq1}
			\de(rs)=\de r\cdot s+r\cdot\de s.
		\end{equation}
	\end{remark}
	
	\begin{remark}
		\label{LD_prop}
		\normalfont
		If $r\in R$, $\alpha\in\Omega_R^1$, and $X,Y\in T_R$ then 
		\begin{align*}
			\LD_X(r\cdot\alpha)(Y)
			&=X((r\cdot\alpha)(Y))-r\alpha([X,Y])\\
			&=X(r\alpha(Y))-r\alpha([X,Y]) \\
			&=X(r)\alpha(Y)+rX(\alpha(Y))-r\alpha([X,Y]) \\
			&=X(r)\alpha(Y)+r(\LD_X\alpha)(Y)
		\end{align*}
		hence 
		\begin{equation}
			\label{LD_prop_eq1}
			\LD_X(r\cdot\alpha)=X(r)\cdot \alpha+r\cdot (\LD_X\alpha).
		\end{equation}
		On the other hand, for every $r\in R$ we have 
		$[rX,Y]=r[X,Y]-Y(r)\cdot X$ 
		(cf. Example~\ref{tg})
		hence 
		\begin{align*}
			(\LD_{r\cdot X}\alpha)(Y)
			&=rX(\alpha(Y))-\alpha([rX,Y])\\
			&=rX(\alpha(Y))-\alpha(r[X,Y]-Y(r)\cdot X)\\
			&=rX(\alpha(Y))-r\alpha([X,Y])+Y(r)\alpha(X) \\
			&=r(\LD_X\alpha)(Y)+\alpha(X)(\de r)(Y)
		\end{align*}
		and then 
		\begin{equation}
			\label{LD_prop_eq2}
			\LD_{r\cdot X}\alpha=r\cdot \LD_X\alpha+\alpha(X)\cdot \de r.
		\end{equation}
	\end{remark}


We now recall from \cite[\S 2]{Ri63} the notion of $(\KK,R)$-Lie algebra, also called 
\emph{Lie-Rinehart algebra} cf. \cite{Hue21} and \cite{Hue90}. 
See	 \cite{Ma95} for more detailed historical remarks on this notion.

\begin{definition}
	\label{LR_def}
	\normalfont
	Assume that we have a unital left $R$-module $L$ with its corresponding structural map denoted by 
	\[R\times L\to L,\quad (r,\alpha)\mapsto r\cdot\alpha \]
	and an $R$-linear map called the \emph{anchor}, denoted by 
	\[\rho\colon L\to T_R.\]  
	A \emph{$(\KK,R)$-Lie algebra}  
	structure on $L$ is a structure of Lie algebra $(L,[\cdot,\cdot])$ over $\KK$ satisfying the following conditions: 
	\begin{enumerate}[{\rm(a)}]
		\item\label{LR_def_item1} 
		The anchor $\rho\colon L\to T_R$ is a homomorphism of Lie algebras. 
		\item\label{LR_def_item2} For all $r\in R$ and $\ell_1,\ell_2\in L$ we have 
		\begin{equation}
			\label{LR_def_eq1}
			[\ell_1,r\cdot\ell_2]=r\cdot [\ell_1,\ell_2]+(\rho(\ell_1)r)\cdot\ell_2.
		\end{equation}
	\end{enumerate}
\end{definition}

\begin{example}[Lie algebroids]
	\normalfont
	If $(A,[\cdot,\cdot]_A,a_A)$ is a Lie algebroid over a manifold $M$ 
	and we denote $R:=\Ci(M,\RR)$, $L:=\Gamma(A)$, and $\rho:=a_A$, 
	then $[\cdot,\cdot]_A$ is a $(\KK,R)$-Lie algebra structure on $L$ with the anchor~$\rho$. 
\end{example}

\begin{example}[abstract tangent bundle]
	\normalfont
	\label{tg}
	If we define $L:=T_R$ and $\rho:=\id_{T_R}$ then $L$ is a $(\KK,R)$-Lie algebra with the anchor~$\rho$. 
	
	In fact, for all $r,s\in R$ and $X_1,X_2\in T_R$  we have 
	\begin{align*}
		[X_1,r\cdot X_2](s)
		&=X_1(r X_2(s))-r X_2(X_1(s)) \\
		&=X_1(r)X_2(s)+r X_1(X_2(s))-r X_2(X_1(s)) \\
		&=(r\cdot [X_1,X_2])(s)+(X_1(r)\cdot X_2)(s) 
	\end{align*}
	hence \eqref{LR_def_eq1} is satisfied. 
\end{example}

\begin{problem}
\label{basic}
	\normalfont 
	Let $\KK$ be a field of characteristic zero and $R$ a unital commutative associative $\KK$-algebra, and  
	assume that $L$ is a unital left $R$-module. 
	\begin{enumerate}[{\rm(A)}]
		\item Characterize the  $R$-linear maps 
		$\rho\colon L\to\Der(R)$ for which there exists a 
		$(\KK,R)$-Lie algebra structure on $L$ with its corresponding anchor~$\rho$. 
		\item For any fixed $R$-linear maps 
		$\rho\colon L\to\Der(R)$, determine all the $(\KK,R)$-Lie algebra structures on $L$ whose anchor is~$\rho$. 
	\end{enumerate}
\end{problem}

\subsection*{Cyclic subalgebras of $T_R$}

We will investigate the  $(\KK,R)$-Lie algebras $L$ which are ``as non-transitive as possible''. 
More specifically, since transitivity means surjectivity of the anchor map (Definition~\ref{transit_def}), we will focus on the case when the image of the anchor map $\rho\colon L\to\Der(R)$ is as small as possible. 
In this sense, the simplest situation is when $\rho=0$, and then \eqref{LR_def_eq1} shows that the Lie bracket of $L$ is $R$-bilinear, 
which corresponds to the algebraic version of Lie algebra bundles. 

As a nontrivial example when $\rho\ne 0$ and $\rho(L)$ is still ``very small'', 
we will consider the case when $\rho(L)$ is a cyclic submodule of the $R$-module $T_R$, in the sense of the following definition: 

\begin{definition}
	\label{cyc_def}
	\normalfont 
	A \emph{cyclic submodule} of $T_R$ is any singly generated submodule of the $R$-module $T_R$, 
	that is, a submodule of the form 
	$$R\cdot X:=\{rX\mid r\in R\}$$ 
	for some $X\in T_R$. 
\end{definition}

\begin{lemma}
	\label{cyc}
	Every cyclic submodule of $T_R$ is a subalgebra of the Lie algebra $T_R$. 
\end{lemma}

\begin{proof}	
	If $X,Y\in T_R=\Der(R)$, then for every $r,s,t\in R$ we have 
	$$rX(sY(t))=rX(s)Y(t)+rs(XY)(t)
	\text{ and }
	sY(rX(t))=sY(r)X(t)+sr(YX)(t).$$ 
	Substracting these equalities from each other and using $rs=sr$, we obtain 
	\begin{equation}
		\label{cyc_proof_eq1}
		[rX,sY]=  rX(s)Y-sY(r)X 
         +rs[X,Y].
	\end{equation}
Then, for $X=Y$, we see that $R\cdot X$ is a subalgebra of the Lie algebra $T_R$.
\end{proof} 

\begin{remark}
	\normalfont 
Equation~\eqref{cyc_proof_eq1} directly implies that if $\gg$ is a subalgebra of the Lie algebra $T_R=\Der(R)$ over $\KK$, the $R$-submodule generated by $\gg$, that is,  $R\cdot\gg:=\spa_\KK\{rX\mid r\in R,X\in\gg\}$, is again a  Lie subalgebra of $T_R$. 
In particular, the $R$-submodule of $T_R$ generated by any finite set of mutually commuting derivations of $R$ is a Lie subalgebra of $T_R$, which is the notion of
 left-symmetric Witt algebra from 
 \cite[Eq. (1.1)]{KCB11}. 
 		See also \cite[Prop. 2.1]{Bu06}.
\end{remark}

\begin{definition}
\normalfont 
In the setting of Definition~\ref{LR_def}, 
if the Lie subalgebra $\rho(L)\subseteq T_R$ is a cyclic submodule of $T_R$, then $L$ is called a \emph{cyclic $(\KK,R)$-Lie algebra} or a \emph{cyclic Lie-Rinehart algebra}.
\end{definition}

\section{Covariant differential operators} 
\label{Sect3}

This section contains some auxiliary facts on covariant differential operators 
that will be needed in Section~\ref{Sect4}. 
Motivated by a problem pointed out in \cite[Rem. 4.19]{BKS22}, 
we also prove that  the Lie-Rinehart algebra associated to a projective module is transitive (Proposition~\ref{transit}). 

\subsection*{On the definition of $\mathcal{CDO}(N)$}

\begin{definition}
	\normalfont
	An \emph{anchored module} is a left $R$-module $N$ endowed with an $R$-linear map $\rho^N\colon N\to T_R$ called \emph{anchor}. 
\end{definition}

\begin{notation}
	\normalfont
	If $N$ is a left $R$-module, then for every $r\in R$ define the multiplication operator $\mu^N_r\colon N\to N$, $\mu^N_r(n):=rn$. 
	We also define the morphism of unital associative $\KK$-algebras 
	$$\mu^N\colon R\to\End_\KK (N),\quad \mu^N(r):=\mu^N_r$$
	and we will regard $\End_\KK(N)$ as an $R$-module via the mapping $\mu^N$, 
	i.e., $r\cdot E:= \mu^N_r E$ for every $r\in R$ and $E\in\End_\KK(N)$. 
\end{notation}

\begin{definition}
	\normalfont
	\label{normalizer}
	Let $N$ be a left $R$-module. 
	We define
	$$\widetilde{\mathcal{CDO}}(N):=\{D\in \End_\KK(N)\mid [D,\mu^N(R)]\subseteq\mu^N(R)\}$$
	that is, $\widetilde{\mathcal{CDO}}(N)$ is the normalizer of $\mu^N(R)$ in the Lie algebra $\End_\KK(N)$. 
	
	We say that the $R$-module $N$ is \emph{nondegenerate} if the mapping $\mu^N$ is injective, 
	that is, for every $r\in R\setminus\{0\}$ there exists $n\in N$ with $rn\ne0$. 
	If this is the case, then for every $D\in \widetilde{\mathcal{CDO}}(N)$ there exists a uniquely determined mapping $X_D\colon R\to R$ satisfying 
	$$(\forall r\in R)\quad [D,\mu^N(r)]=\mu^N(X_D(r)),$$
	that is, $D(rn)-rD(n)=X_D(r)n$ for all $r\in R$ and $n\in N$. 
\end{definition}

\begin{lemma}
	If $N$ is an $R$-module, then the following assertions hold: 
	\begin{enumerate}[{\rm(i)}]
		\item The set $\widetilde{\mathcal{CDO}}(N)$ is both a Lie subalgebra and a sub-$R$-module of $\End_\KK(N)$. 
		\item If the $R$-module $N$ is nondegenerate, then  for every $D\in \widetilde{\mathcal{CDO}}(N)$ we have $X_D\in\Der(R)$ and the mapping 
		$$\rho^{\widetilde{\mathcal{CDO}}(N)}\colon \widetilde{\mathcal{CDO}}(N)\to\Der(R),\quad D\mapsto X_D$$
		is both a Lie algebra morphism and a morphism of $R$-modules. 
	\end{enumerate}
\end{lemma}

\begin{proof}
	All the assertions are straightforward and well known. 
	See for instance \cite[\S 1.2]{Ko76}, where the elements of $\widetilde{\mathcal{CDO}}(N)$ are called derivative endomorphisms of~$N$.
\end{proof}

The case of general (not necessarily nondegenerate) $R$-modules is handled in the following way, cf. \cite[(2.11)]{Hue90}, \cite[\S 2.2]{CLP04}, and \cite[Sect. 2, Ex. 1(vi)]{C11}: 

\begin{definition}
	\normalfont 
	\label{CDO_def}
	For every $R$-module $N$ we define 
	$$\mathcal{CDO}(N):=\{(D,X)\in\End_\KK(N)\times\Der(R)\mid (\forall r\in R)\ 
	[D,\mu^N(r)]=\mu^N(X(r))\}$$
	and 
	$$\rho^{\mathcal{CDO}(N)}\colon \mathcal{CDO}(N)\to\Der(R),\quad 
	\rho^{\mathcal{CDO}(N)}(D,X):=X.$$
	Then $\mathcal{CDO}(N)$ is both a Lie subalgebra and a sub-$R$-module of $\End_\KK(N)\times\Der(R)$.
\end{definition}

\begin{remark}
	\normalfont 
	\label{CDO_rem}
	For an arbitrary $R$-module $N$, the set $\mathcal{CDO}(N)$ is both a Lie subalgebra and a sub-$R$-module of $\End_\KK(N)\times\Der(R)$, the mapping $\rho^{\mathcal{CDO}(N)}\colon \mathcal{CDO}(N)\to\Der(R)$ is both a Lie algebra morphism and a morphism of $R$-modules,  and moreover $\mathcal{CDO}(N)$ is a $(\KK,R)$-Lie algebra with its anchor $\rho^{\mathcal{CDO}(N)}$. 
	
	If moreover the $R$-module $N$ is nondegenerate, then the mapping 
	$$\widetilde{\mathcal{CDO}}(N)\to\mathcal{CDO}(N),\quad D\mapsto (D,X_D)$$ 
	is an isomorphism of $(\KK,R)$-Lie algebras with its inverse 
	$$\mathcal{CDO}(N)\to\widetilde{\mathcal{CDO}}(N),\quad(D,X)\mapsto D.$$ 
\end{remark}

\begin{definition}
	\label{conn_def}
	\normalfont 
	If $L$ and $M$ are anchored modules, then an \emph{$M$-connection} in $L$ 
	is an $R$-linear mapping $\nabla\colon M\to L$ satisfying $\rho^L\circ\nabla=\rho^M$. 
	
	In the special case when $L=\mathcal{CDO}(N)$ for some $R$-module $N$, 
	we write 
	$$\nabla(m)=(\nabla_m,\rho^M(m))$$ 
	where $\nabla_m\in\widetilde{\mathcal{CDO}}(N)$. 
\end{definition}

\subsection*{On the existence of connections}

\begin{definition}
	\label{transit_def}
	\normalfont 
	A $(\KK,R)$-Lie algebra $L$ is called \emph{transitive} if its anchor $\rho^L\colon L\to T_R$ is surjective. 
\end{definition}

We recall that the Lie algebroid of any vector bundle is transitive. 
(See for instance \cite[Ch. I, \S 1.3]{Ba21}.)
The transitivity problem for the $(\KK,R)$-Lie algebra of an $R$-module is pointed out in \cite[Rem. 4.19]{BKS22}. 
In the following proposition we show that the transitivity problem has an affirmative answer in the geometrically relevant case of projective modules. 

\begin{proposition}
	\label{transit}
	If $N$ is a projective left $R$-module, then its corresponding $(\KK,R)$-Lie algebra $\mathcal{CDO}(N)$ is transitive. 
\end{proposition}

\begin{proof}
	We must prove that for every $X\in T_R$ there exists 
$D\in\widetilde{\mathcal{CDO}}(N)$ with $(D,X)\in\mathcal{CDO}(N)$.
	
	Since $N$ is a projective $R$-module, there exist a free left $R$-module $F$ and an $R$-linear map $P\colon F\to F$ satisfying $P\circ P=P$, whose range $P(F)$ is isomorphic to $N$ as $R$-modules. 
	Without loss of generality we may assume $P(F)=N$. 
	
	Let $(e_i)_{i\in I}$ be a basis of the free $R$-module $F$. 
	Then there exists a unique family $(e_i^*)_{i\in I}$ in $\Hom_R(F,R)$ with the property that, for every $f\in F$, the set $\{i\in I\mid e_i^*(f)\ne0\}$ is finite and 
	$f=\sum\limits_{i\in I}e_i^*(f)e_i$.  
	We then define 
	$$\widetilde{D}\colon F\to F,\quad \widetilde{D}(f):=\sum_{i\in I}X(e_i^*(f))\cdot P(e_i).$$
	For every $f\in F$ and $r\in R$ we have 
	\begin{align*}
		\widetilde{D}(rf)
		&=\sum_{i\in I}X(e_i^*(rf))\cdot P(e_i) \\
		&=\sum_{i\in I}X(re_i^*(f))\cdot P(e_i) \\
		&=\sum_{i\in I}\bigl(rX(e_i^*(f))+X(r)e_i^*(f)\bigr) P(e_i) \\
		&=r\widetilde{D}(f)+X(r)P(f).
	\end{align*}
	On the other hand, since $P\circ P=P$ and $P(F)=N$, we have $P(f)=f$ for every $f\in N$, 
	hence the above equality implies 
	$$(\forall f\in N)(\forall r\in R)\quad \widetilde{D}(rf)=r\widetilde{D}(f)+X(r)f.$$
	We now note that, since $P(F)=N$, the definition of $\widetilde{D}$ implies $\widetilde{D}(F)\subseteq N$. 
	Consequently, if we define $D:=\widetilde{D}\vert_N$, then we have $D\colon N\to N$ 
	and moreover 
	$(D,X)\in\mathcal{CDO}(N)$,  
	which completes the proof. 
\end{proof}

We now draw the following corollary on the existence of connections can be established in the present abstract framework, in which the partitions of unity are not available as in the classical setting. 
This result was first obtained in \cite[Prop. 1.1]{Po06}.
(See also \cite[(2.14)]{Hue90}.)

\begin{corollary}
	Let $L$ be an anchored projective $R$-module. 
	For every projective left $R$-module $N$  
	there exists a connection $\nabla\colon L\to\mathcal{CDO}(N)$. 
\end{corollary}

\begin{proof}
	Since the $R$-module $N$ is projective, we know from Proposition~\ref{transit} that the anchor mapping $\rho^{\mathcal{CDO}(N)}\colon \mathcal{CDO}(N)\to T_R$ is surjective. 
	Since $L$ is a projective $R$-module, it then follows that there exists an $R$-linear mapping $\nabla\colon L\to\mathcal{CDO}(N)$ satisfying  $\rho^{\mathcal{CDO}(N)}\circ\nabla=\rho^L$. 
\end{proof}

\section{Duality pairings and pre-Lie-Rinehart algebra structures}
\label{Sect4}

This section contains our main results, which provide new examples of cyclic Lie-Rinehart algebras. 
See Theorems \ref{LSA_prop} and \ref{LSA_prop_new} as well as their corollaries.

\subsection*{Pre-Lie-Rinehart algebras}

We recall the notion of pre-Lie-Rinehart algebra in the sense of \cite[Def. 2.9]{FMaMu21}; 
see also \cite[Def. 2.8]{ChLL22}. 
In differential geometry, this notion corresponds to the left-symmetric algebroids studied in \cite{LShBCh16}.

\begin{definition}
	\label{preLR_def}
	\normalfont
	Assume that we have a unital left $R$-module $L$ which is an anchored module with its corresponding anchor $\rho^L\colon L\to T_R$. 
	A 
	\emph{pre-Lie-Rinehart algebra} structure on $L$ is an $R$-linear map
	(i.e., an $L$-connection) 
	$$\nabla\colon L\to\mathcal{CDO}(L), \quad \ell\mapsto \nabla(\ell)=(\nabla_\ell,\rho^L(\ell)),$$ 
	where $\nabla_\ell\in\widetilde{\mathcal{CDO}}(L)$ for every $\ell\in L$, 
	 satisfying the following conditions, 
	 in which we define 
	 \begin{equation}
	 	\label{preLR_def_eq1} 
	 (\forall \ell_1,\ell_2\in L)\quad \ell_1\cdot\ell_2:=\nabla_{\ell_1}\ell_2 
	 \text{ and }[\ell_1,\ell_2]:= \ell_1\cdot\ell_2- \ell_2\cdot\ell_1.
	 \end{equation}
	\begin{enumerate}[{\rm(a)}]
		\item\label{preLR_def_item1} 
		The associator mapping 
		\begin{equation}
			\label{preLR_def_item1_eq1} 
		(\cdot,\cdot,\cdot)\colon L\times L\times L\to L, \quad 
			(\ell_1,\ell_2,\ell_3)
		:=
		\ell_1\cdot(\ell_2\cdot \ell_3)-(\ell_1\cdot \ell_2)\cdot \ell_3
		\end{equation}
	is \emph{left symmetric}, i.e., it satisfies $(\ell_1,\ell_2,\ell_3)=(\ell_2,\ell_1,\ell_3)$ for all $\ell_1,\ell_2,\ell_3\in L$. 
	\item\label{preLR_def_item2} We have 
	\begin{equation}
		\label{preLR_def_item2_eq1} 
		(\forall \ell_1,\ell_2\in L)\quad 
		\rho^L([\ell_1,\ell_2])
		=[\rho^L(\ell_1),\rho^L(\ell_2)]\in \Der(R)=T_R. 
	\end{equation}
	\end{enumerate}
\end{definition}

\begin{remark}
\label{preLR_rem}	
\normalfont
Here are two simple remarks on Definition~\ref{preLR_def}: 
\begin{enumerate}[{\rm(i)}]
	\item\label{preLR_rem_item1} 
The bracket $[\cdot,\cdot]$ defined in equation~\eqref{preLR_def_eq1}  
	turns the anchored module $L$ into a Lie-Rinehart algebra. 	
	\\
	In fact, it suffices to check that  bracket $[\cdot,\cdot]$ satisfies the Jacobi identity. 
	To this end, one needs the remark at the bottom of \cite[p. 324]{Bu06}, 
	which gives the following relation between the Jacobiator and the associator: 
	\begin{align*}
		[\ell_1,[\ell_2,\ell_3]]
		& +[\ell_2,[\ell_3,\ell_1]]+[\ell_3,[\ell_1,\ell_2]]\\
		=
		& (\ell_1,\ell_2,\ell_3)+(\ell_2,\ell_3,\ell_1)+(\ell_3,\ell_1,\ell_2) \\
		&-(\ell_2,\ell_1,\ell_3)-(\ell_3,\ell_2,\ell_1)-(\ell_1,\ell_3,\ell_2). 
	\end{align*}
Thus, the bracket  $[\cdot,\cdot]$ satisfies the Jacobi identity 
	if the associator 
	$(\ell_1,\ell_2,\ell_3)$ 
	is symmetric in its first two variables. 
	\item\label{preLR_rem_item2} 
	It is easily seen that the left-symmetry property of the associator mapping is equivalent to 
	$$ 	(\forall\ell_1,\ell_2\in L)\quad	\nabla_{\ell_1\cdot\ell_2- \ell_2\cdot\ell_1}
	=\nabla_{\ell_1}\nabla_{\ell_2}-\nabla_{\ell_2}\nabla_{\ell_1}\in\End_\KK(L),$$
	Hence the conditions \eqref{preLR_def_item1}--\eqref{preLR_def_item2} in Definition~\ref{preLR_def} equation~\eqref{preLR_def_item2_eq1}  are equivalent to 
	\begin{equation}
		\label{preLR_rem_item2_eq1} 
	(\forall\ell_1,\ell_2\in L)\quad \nabla([\ell_1,\ell_2])=[\nabla(\ell_1),\nabla(\ell_2)].
	\end{equation}
In the right-hand side of equation~\eqref{preLR_rem_item2_eq1} we use the Lie bracket of the Lie-Rinehart algebra $\mathcal{CDO}(L)$, 	cf. Definition~\ref{CDO_def} and Remark~\ref{CDO_rem}. 
In the terminology of \cite[\S 2.1]{FMaMu21}, the equation~\eqref{preLR_rem_item2_eq1} means that the connection~$\nabla$ is flat. 
\end{enumerate}
\end{remark}

\subsection*{Duality pairings}

\begin{definition}
	\label{duality_def}
	\normalfont 
	A \emph{duality pairing} of two $R$-modules $L$ and $N$ is an $R$-bilinear mapping 
	$\langle\cdot,\cdot\rangle\colon L\times N\to R$ 
	satisfying the following nondegeneracy conditions for any $\ell\in L$ and $n\in N$: 
	\begin{equation}
		\label{duality_def_l}
		\ell =0\iff\langle \ell ,N\rangle=\{0\}
	\end{equation}
	and 
	\begin{equation}
		\label{duality_def_r}
		n=0\iff\langle L,n\rangle=\{0\}.
	\end{equation}
	If this is the case, and $E\in\End_\KK(L)$ and $D\in\End_\KK(N)$, 
	then we say that \emph{the endomorphisms $D$ and $E$ are dual to each other} if 
	$$(\forall \ell\in L)(\forall n\in N)\quad \langle E(\ell),n\rangle=\langle \ell,D(n)\rangle.$$
\end{definition}

\begin{remark}
	\label{nondegs}
	\normalfont
	If $\langle\cdot,\cdot\rangle\colon L\times N\to R$  is a duality pairing, 
	then the $R$-module $L$ is nondegenerate if and only if $N$ is nondegenerate. 
	(See Definition~\ref{normalizer}.)
	
	Indeed, let us assume that $L$ is nondegenerate and $r_0\in R$ satisfies $r_0n=0$ for every $n\in N$. 
	Then for every $\ell\in L$ and $n\in N$ we have $\langle r_0 \ell,n\rangle=r_0\langle \ell,n\rangle=\langle \ell,r_0n\rangle=0$. 
	Hence, by \eqref{duality_def_l}, 
	we obtain $r_0\ell=0$ for every $\ell\in L$ and then, since the $R$-module $L$ is non-degenerate, we obtain $r_0=0$. 
	This shows that the $R$-module $N$ is nondegenerate, too. 
\end{remark}

\begin{example}
	\normalfont 
	If  $\langle\cdot,\cdot\rangle\colon L\times N\to R$  is a duality pairing of  two $R$-modules $L$ and $N$, then for every $r\in R$ the multiplication-by-$r$ operators $\mu^L_r \colon L\to L$ and $\mu^N_r \colon N\to N$ are dual to each others 
	since we have 
	$\langle r\ell,n\rangle=r\langle \ell,n\rangle=\langle \ell,rn\rangle$ for all $\ell\in L$ and $n\in N$. 
\end{example}

\begin{example}
	\normalfont 
	If $N$ is an $R$-module and we define $L:=\Hom_R(N,R)$, then the evaluation mapping 
	$$\langle\cdot,\cdot\rangle\colon L\times N\to R,\quad 
	\langle\alpha,n\rangle:=\alpha(n)$$
	is a duality pairing. 
\end{example}

\begin{lemma}
	\label{LSA_lemma}
	Assume that $\langle\cdot,\cdot\rangle\colon L\times N\to R$  is a duality pairing of  two $R$-modules $L$ and $N$, and moreover we have two endomorphisms $E\in\End_\KK(L)$ and $D\in\End_\KK(N)$ 
	satisfying
	$$(\forall \ell\in L)(\forall n\in N)\quad 
	T(\langle \ell,n\rangle)=\langle E(\ell),n\rangle+\langle \ell,D(n)\rangle
	$$
	for some mapping $T\colon R\to R$. 
	
	Then for every $X\in T_R$ we have 
	$$(E,X)\in\mathcal{CDO}(L)\iff (D,X)\in\mathcal{CDO}(N).$$
\end{lemma}

\begin{proof}
	It suffices to prove the implication``$\Rightarrow$''.
	The hypothesis $(E,X)\in\mathcal{CDO}(L)$ implies 
	\begin{equation*}
		(\forall \ell\in L)	\quad E(r\ell)=rE(\ell)+X(r)\ell.
	\end{equation*}
	We then obtain for every $n\in N$ 
	\begin{align*}
		\langle \ell,D(rn)\rangle
		& =-\langle E(\ell),rn\rangle +T(\langle \ell,rn\rangle)\\
		& =-\langle rE(\ell),n\rangle +T(\langle r\ell,n\rangle)\\
		&=-\langle E(r\ell)-X(r)\ell,n\rangle+T(\langle r\ell,n\rangle) \\
		&=\langle r\ell,D(n)\rangle+X(r)\langle \ell,n\rangle\\
		&=\langle \ell,rD(n)+X(r)n\rangle.
	\end{align*}
	Now, using \eqref{duality_def_r}, we obtain 
	\begin{equation*}
		(\forall n\in N)\quad D(rn)=rD(n)+X(r)n
	\end{equation*} 
	hence $(D,X)\in\mathcal{CDO}(N)$.
\end{proof}

\begin{theorem}
	\label{LSA_prop}
	Assume that we have the following data: 
	\begin{itemize}
		\item $\langle\cdot,\cdot\rangle\colon L\times N\to R$  is a duality pairing of  two $R$-modules $L$ and $N$;
		\item $(D_0,X_0)\in \mathcal{CDO}(N)$, $y_0\in N$, and $E_0\in\End_\KK(L)$
		satisfying the conditions that 
		the mapping 
		\begin{equation}
			\label{LSA_prop_eq1}
			L\times L\to R,\quad (\ell_1,\ell_2)\mapsto \langle \ell_1,y_0\rangle 
			\langle \ell_2,D_0(y_0)\rangle
		\end{equation}
	is symmetric  
	and 
		\begin{equation}
			\label{LSA_prop_eq2}
			(\forall \ell\in L)(\forall n\in N)\quad 
			X_0(\langle \ell,n\rangle)
			=\langle E_0(\ell),n\rangle+\langle \ell,
			D_0(n)\rangle.
		\end{equation}
	\end{itemize}
	If we define the anchor 
	\begin{equation}
		\label{LSA_prop_eq3}
		\rho^L\colon L\to T_R,\quad \rho^L(\ell):=\langle \ell,y_0\rangle 
		X_0
	\end{equation}
	then the mapping 
	$$\nabla\colon L\to \mathcal{CDO}(L),\quad \nabla(\ell):=(\langle \ell,y_0\rangle E_0,\rho^L(\ell))
	\in\widetilde{\mathcal{CDO}}(L)\times T_R
	$$
	is a well-defined $L$-connection on the anchored module $L$ 
	and the mapping 
	\begin{equation}
	\label{multiplication-1}
    L\times L\to L,\quad (\ell_1,\ell_2)\mapsto \ell_1\cdot \ell_2:=\nabla_{\ell_1}\ell_2=\langle \ell_1,y_0\rangle E_0(\ell_2)
    \end{equation}
	defines the structure of a pre-Lie-Rinehart algebra on $L$. 
\end{theorem}

\begin{proof}	
	Since 
		$(D_0,X_0)\in \mathcal{CDO}(N)$, 
	it follows by Lemma~\ref{LSA_lemma} that 
	$E_0\in\widetilde{\mathcal{CDO}}(L)$ 
	and 
	\begin{equation}
		\label{LSA_prop_proof_eq1}
		(\forall \ell\in L)(\forall r\in R)\quad E_0(r\ell)=rE_0(\ell)+
		X_0(r)\ell.
	\end{equation}
	It also follows by \eqref{LSA_prop_eq3} that the mapping $\nabla$ is 	$R$-linear. 
	Moreover, by \eqref{LSA_prop_eq2}, we have 
	$\nabla(\ell)\in\mathcal{CDO}(L)$ for every $\ell\in L$. 
	Thus, since $\rho^{\mathcal{CDO}(L)}\circ\nabla=\rho^L$, we see that 
	$\nabla$ is a well-defined $L$-connection on the anchored module $L$. 
	
	For later use, we now note that the symmetry hypothesis on \eqref{LSA_prop_eq1} is equivalent to 
	\begin{equation}
		\label{LSA_prop_proof_eq2}
		(\forall \ell_1,\ell_2\in L) \quad 
		\langle \ell_1,y_0\rangle \langle \ell_2,
		D_0(y_0)\rangle
		=\langle \ell_2,y_0\rangle \langle \ell_1, 
		D_0(y_0)\rangle. 
	\end{equation}
	For any $\ell_1,\ell_2,\ell_3\in L$ we have 
	\begin{align*}
		(\ell_1,\ell_2,\ell_3)
		:=
		&\ell_1\cdot(\ell_2\cdot \ell_3)-(\ell_1\cdot \ell_2)\cdot \ell_3 \\
		=
		&\langle \ell_1,y_0\rangle E_0(\ell_2\cdot \ell_3)-\langle \ell_1\cdot \ell_2,y_0\rangle E_0(\ell_3)\\
		=
		&\langle \ell_1,y_0\rangle E_0( \langle \ell_2,y_0\rangle E_0(\ell_3)) \\
		&-\langle \ell_1,y_0\rangle\langle E_0(\ell_2),y_0\rangle E_0(\ell_3) \\
		=&
		\langle \ell_1,y_0\rangle  \langle \ell_2,y_0\rangle E_0(E_0(\ell_3))
		+ \langle \ell_1,y_0\rangle  
		X_0	(\langle \ell_2,y_0\rangle) E_0(\ell_3) \\
		&-\langle \ell_1,y_0\rangle 
		X_0(\langle \ell_2,y_0\rangle) E_0(\ell_3)
		+\langle \ell_1,y_0\rangle \langle \ell_2,
		D_0y_0\rangle  E_0(\ell_3) \\
		=
		&\langle \ell_1,y_0\rangle  \langle \ell_2,y_0\rangle E_0(E_0(\ell_3))
		+\langle \ell_1,y_0\rangle \langle \ell_2,
		D_0y_0\rangle  E_0(\ell_3),
	\end{align*}
	where, in the next-to-last equality we used \eqref{LSA_prop_proof_eq1} and \eqref{LSA_prop_eq2}. 
	Now, using \eqref{LSA_prop_proof_eq2}, 
	we see that the associator 
	$(\ell_1,\ell_2,\ell_3)$ is symmetric with respect to $\ell_1$ and $\ell_2$, and this completes the proof.
	
	Finally, we check compatibility with the anchor $\rho^L$: 
	\begin{align*}
		\rho^L([\ell_1,\ell_2]) =  &\rho^L(\ell_1\cdot  \ell_2-\ell_2\cdot \ell_1)\\
		=
		&\langle \ell_1\cdot \ell_2-\ell_2\cdot \ell_1,y_0\rangle
		X_0\\
		=
		&(\langle \ell_1,y_0\rangle\langle E_0(\ell_2),y_0\rangle-
		\langle \ell_2,y_0\rangle\langle E_0(\ell_1),y_0\rangle)
		X_0\\
		\mathop{=}\limits^{\eqref{LSA_prop_proof_eq2}}
		&(\langle \ell_1,y_0\rangle \langle E_0(\ell_2),y_0\rangle
		+ \langle \ell_1,y_0\rangle \langle \ell_2,
		D_0y_0\rangle\\
		&-\langle \ell_2,y_0\rangle \langle \ell_1,
		D_0y_0\rangle
		-\langle \ell_2,y_0\rangle\langle E_0(\ell_1),y_0\rangle)
		X_0\\
		=
		&(\langle \ell_1,y_0\rangle 
	    X_0(\langle \ell_2,y_0\rangle)
		-\langle \ell_2,y_0\rangle
		X_0(\langle \ell_1,y_0\rangle)) 
		X_0\\
		\mathop{=}\limits^{\eqref{cyc_proof_eq1}}
		&
		[\rho^L(\ell_1),\rho^L(\ell_2)].
	\end{align*}
	This completes the proof. 
\end{proof}

\begin{corollary}
	\label{LSA_cor1}
	In the setting of Theorem~\ref{LSA_prop}, 
	$L$ is a $(\KK,R)$-Lie algebra with respect to 
	the bracket defined by 
	\begin{equation}
	\label{LSA_cor1_eq1}
	[\ell_1,\ell_2]:=\nabla_{\ell_1}\ell_2-\nabla_{\ell_2}\ell_1
	\end{equation}
	for all $\ell_1,\ell_2\in L$. 
\end{corollary}

\begin{proof}
	Use Theorem~\ref{LSA_prop},  
	Definition~\ref{preLR_def}, 
	and Remark~\ref{preLR_rem}. 
\end{proof}

\begin{remark}
	\label{eigenvalue}
	\normalfont 
	Condition~\eqref{LSA_prop_proof_eq2} is satisfied if for instance 
	$D_0y_0=cy_0$ 
	for some element 
	$c\in R$. 
\end{remark}

\begin{remark}
	\normalfont 
	In this way we obtain  $(\KK,R)$-Lie algebras of a special type, called pre-Lie-Rinehart algebras in \cite[Def. 2.9]{FMaMu21}. 
	See also \cite[Def. 2.8]{ChLL22}. 
	The differential geometric counterpart of this notion had been investigated in \cite{LShBCh16} under the name of left-symmetric algebroid.
\end{remark}

We now apply Theorem~\ref{LSA_prop} to the $R$-bilinear duality map $\langle\cdot,\cdot\rangle\colon\Omega^1_R\times T_R\to R$. 

\begin{corollary}
	\label{first}
	For the $R$-module $\Omega^1_R$ and $X,Y\in T_R$, 
	we define the anchor map 
	\begin{equation}
		\label{first_eq1}
		\rho\colon \Omega^1_R\to T_R, \quad \rho(\alpha):=\langle\alpha,Y\rangle X.
	\end{equation}
	If the $R$-bilinear mapping 
	\begin{equation}
		\label{first_eq2}
		\Omega^1_R\times \Omega^1_R\to R,\quad (\alpha_1,\alpha_2)\mapsto \langle\alpha_1,Y\rangle\langle\alpha_2,[X,Y]\rangle
	\end{equation}
	is symmetric, then the mapping 
	$$\Omega^1_R\times \Omega^1_R\to \Omega^1_R,\quad (\alpha_1,\alpha_2)\mapsto \alpha_1\cdot\alpha_2:=\nabla_{\alpha_1}\alpha_2:=\langle\alpha_1,Y\rangle\LD_X\alpha_2 
	$$
	defines the structure of a pre-Lie-Rinehart algebra on $\Omega^1_R$, 
	with its corresponding Lie-Rinehart bracket 
	\begin{equation}
		\label{first_eq3}
		\Omega^1_R\times \Omega^1_R\to \Omega^1_R,\quad [\alpha_1,\alpha_2]:=\langle\alpha_1,Y\rangle\LD_X\alpha_2 - \langle\alpha_2,Y\rangle\LD_X\alpha_1.
	\end{equation}
\end{corollary}

\begin{proof}
	We consider the $R$-modules $L:=\Omega^1_R$ 
	and $N:=T_R$ 
	and 
	$$D_0:=\ad_{T_R}X\colon T_R\to T_R$$ 
	where $\ad_{T_R}X\colon T_R\to T_R$, 
	$(\ad_{T_R}X)(Y):=[X,Y]$,
is the adjoint representation of the Lie algebra $T_R=\Der(R)$.  
	Then the equation~\eqref{LSA_prop_eq2} is satisfied with the Lie derivative $$E_0:=\LD_X\colon \Omega^1_R\to\Omega^1_R$$ 
	 (condition~\eqref{LSA_prop_eq2} in this case comes down to the definition of the Lie derivative~\eqref{LD1})
	hence the assertion follows by Theorem~\ref{LSA_prop} and equation~\eqref{LSA_prop_proof_eq2}, 
	applied for $X_0:=X$ and $y_0:=Y$. 
\end{proof}

The following corollary is suggested by Remark~\ref{eigenvalue}. 

\begin{corollary}
	\label{first0}
	If $X,Y\in T_R$ and there exists $c\in R$ satisfying 
	\begin{equation}
		\label{first0_eq1}
		[X,Y]=cY
	\end{equation}
	then the bracket \eqref{first_eq3} gives a Lie-Rinehart algebra structure on $\Omega^1_R$ with the anchor~$\rho$ given by~\eqref{first_eq1}. 
\end{corollary}

\begin{proof}
	If \eqref{first0_eq1} is satisfied, then it is easily seen that the mapping \eqref{first_eq2} is symmetric, hence the assertion follows by Corollary~\ref{first}. 
\end{proof}

The examples of Lie-Rinehart brackets provided by Corollary~\ref{first0}  were studied  for Lie-Rinehart algebras associated to Lie algebroids in \cite{DJ21}, see also \cite{DJS22}.


\begin{theorem}
	\label{LSA_prop_new}
	Assume that we have the following data: 
	\begin{itemize}
		\item $\langle\cdot,\cdot\rangle\colon L\times N\to R$  is a duality pairing of  two $R$-modules $L$ and $N$.
\item $(D_0,X_0)\in\mathcal{CDO}(N)$, $y_0\in N$, and $E_0\in\End_\KK(L)$ 
satisfying the conditions that the mapping 
		\begin{equation}
			\label{LSA_prop_new_eq1}
			L\times L\to R, \quad (\ell_1,\ell_2)\mapsto \langle\ell_1,y_0\rangle \langle\ell_2,
			(D_0)^2
			(y_0)\rangle\text{ is symmetric }
		\end{equation}
		and 
		\begin{equation}
			\label{LSA_prop_new_eq2}
			(\forall \ell\in L)(\forall n\in N)\quad 
			X_0(\langle \ell,n\rangle)
			=\langle E_0(\ell),n\rangle+\langle \ell,
			D_0(n)\rangle.
		\end{equation}
	\end{itemize}
	If we define the anchor 
	\begin{equation}
		\label{LSA_prop_new_eq3}
		\rho^L\colon L\to T_R,\quad \rho^L(\ell):=\langle \ell,y_0\rangle 
		X_0
	\end{equation}
	then the mapping 
	$$\nabla\colon L\to \mathcal{CDO}(L),\quad \nabla(\ell):=
	(\langle \ell,y_0\rangle E_0-\langle\ell,
	D_0(y_0)\rangle\id_L,\rho^L(\ell))
	\in\widetilde{\mathcal{CDO}}(L)\times T_R
	$$
	is a well-defined $L$-connection on the anchored module $L$ with its anchor $\rho^L$ 
	and the mapping 
	\begin{equation}
	\label{multiplication-2}
    L\times L\to L,\quad (\ell_1,\ell_2)\mapsto \ell_1\cdot \ell_2:=\nabla_{\ell_1}\ell_2= \langle \ell_1,y_0\rangle E_0(\ell_2)-\langle\ell_1,
	D_0 (y_0)\rangle\ell_2
	\end{equation}
	defines the structure of a pre-Lie-Rinehart algebra on $L$. 
\end{theorem}

\begin{proof}
We first check that $\nabla$ is a well-defined $L$-connection on the anchored module~$L$. 
To this end we first note that, since $(D_0,X_0)\in\mathcal{CDO}(N)$, 
it follows by the hypothesis~\ref{LSA_prop_new_eq2} along  with Lemma~\ref{LSA_lemma} that 
\begin{equation}
	\label{LSA_prop_new_proof_eq1}
(\forall\ell\in L)(\forall r\in R) \quad E_0(r\ell)=
X_0(r)\ell+r E_0(\ell). 
\end{equation}
Now, for every $\ell_1,\ell_2\in L$  and $r\in R$ we obtain 
\begin{align*}
\nabla_{\ell_1}(r\ell_2)
&=\langle\ell_1,y_0\rangle
E_0(r\ell_2) -r\langle\ell_1,
D_0(y_0)\rangle\ell_2 \\
& \mathop{=}\limits^{\eqref{LSA_prop_new_proof_eq1}}
\langle\ell_1,y_0\rangle 
X_0(r)\ell_2
+r\langle\ell_1,y_0\rangle E_0(\ell_2)
-r\langle\ell_1,
D_0(y_0)\rangle\ell_2 \\
&=r\nabla_{\ell_1}\ell_2+(\rho^L(\ell_1))(r)\ell_2
\end{align*}
which shows that $\nabla(\ell_1)=(\nabla_{\ell_1},\rho^L(\ell_1))\in\mathcal{CDO}(L)$. 
It is clear that the mapping $\ell\mapsto\nabla(\ell)$ is $R$-linear, hence 
$\nabla $ is indeed a well-defined $L$-connection on the anchored module $L$ with its anchor $\rho^L$. 

Moreover, using the notation 
$\ell_1\cdot\ell_2=\nabla_{\ell_1}\ell_2$, 
 we have
\allowdisplaybreaks
\begin{align*}
	 (\ell_1,\ell_2,\ell_3 )=&\ell_1\cdot  (\ell_2\cdot\ell_3)-(\ell_1\cdot\ell_2)\cdot\ell_3 \\
	=& 
	\nabla_{\ell_1}(\nabla_{\ell_2}\ell_3)-\nabla_{\nabla_{\ell_1}\ell_2}\ell_3 \\
	=&
	\langle\ell_1,y_0\rangle E_0(\langle\ell_2,y_0\rangle E_0(\ell_3) -\langle\ell_2,
	D_0(y_0)\rangle\ell_3) \\
	&-\langle\ell_1,
	D_0(y_0)\rangle(\langle\ell_2,y_0\rangle E_0(\ell_3) -\langle\ell_2,
	D_0(y_0)\rangle\ell_3) \\
	&-\langle\langle\ell_1,y_0\rangle E_0(\ell_2) -\langle\ell_1,
	D_0(y_0)\rangle\ell_2,y_0\rangle 
	E_0(\ell_3) \\
	&+\langle\langle\ell_1,y_0\rangle E_0(\ell_2) -\langle\ell_1,
	D_0(y_0)\rangle\ell_2,
	D_0(y_0)\rangle 
	\ell_3 \\
	\mathop{=}\limits^{\eqref{LSA_prop_new_proof_eq1}}
	& \langle\ell_1,y_0\rangle 
	\Bigl(\langle\ell_2,y_0\rangle 
	(E_0)^2(\ell_3) 
	+
	X_0(\langle\ell_2,y_0\rangle)E_0(\ell_3) 
	- \langle\ell_2,
	D_0(y_0)\rangle E_0(\ell_3) \\
	& -
	X_0(\langle\ell_2,
	D_0(y_0)\rangle)\ell_3\Bigr)\\
	&-\langle\ell_1,
	D_0(y_0)\rangle\langle\ell_2,y_0\rangle E_0(\ell_3)
	+\langle\ell_1,
	D_0(y_0)\rangle \langle\ell_2,
	D_0(y_0)\rangle\ell_3 \\
	&-\langle\ell_1,y_0\rangle \langle E_0(\ell_2),y_0\rangle E_0(\ell_3)   +\langle\ell_1,
	D_0(y_0)\rangle \langle\ell_2,y_0\rangle 
	E_0(\ell_3) 
	\\
	&+\langle\ell_1,y_0\rangle \langle E_0(\ell_2),
	D_0(y_0)\rangle\ell_3   
	-\langle\ell_1,
	D_0(y_0)\rangle \langle\ell_2,
	D_0(y_0)\rangle 
	\ell_3 
	\\
	=& \langle\ell_1,y_0\rangle \langle\ell_2,y_0\rangle (E_0)^2(\ell_3) 
	+\langle\ell_1,y_0\rangle 
	X_0(\langle\ell_2,y_0\rangle) E_0(\ell_3) \\
	&-\langle\ell_1,y_0\rangle \langle\ell_2,
	D_0(y_0)\rangle E_0(\ell_3)\\
	&-\langle\ell_1,y_0\rangle 
	X_0(\langle\ell_2,
	D_0(y_0)\rangle)\ell_3 
	-\langle\ell_1,y_0\rangle  \langle E_0(\ell_2),y_0\rangle E_0(\ell_3)   
	\\
	&+\langle\ell_1,y_0\rangle\langle E_0(\ell_2),
	D_0(y_0)\rangle\ell_3   
	\\
	=&\langle\ell_1,y_0\rangle \langle\ell_2,y_0\rangle (E_0)^2(\ell_3) \\
	&+\langle\ell_1,y_0\rangle \Bigl(
	X_0(\langle\ell_2,y_0\rangle)
	-\langle E_0(\ell_2),y_0\rangle
	-\langle\ell_2,
	D_0(y_0)\rangle\Bigr)E_0(\ell_3)
	\\
	&-\langle\ell_1,y_0\rangle 
	\Bigl( 
	X_0( \langle\ell_2,
	D_0(y_0)\rangle)-\langle E_0(\ell_2),
	D_0(y_0)\rangle\Bigr)\ell_3. 
\end{align*}
By \eqref{LSA_prop_new_eq2} we have 
$$
X_0(\langle\ell_2,y_0\rangle)
 =\langle E_0 (\ell_2),y_0\rangle+\langle\ell_2,
 D_0(y_0)\rangle$$
and 
$$
X_0(\langle\ell_2,
D_0(y_0)\rangle)
=\langle E_0(\ell_2),
D_0(y_0)\rangle+\langle\ell_2,
(D_0)^2(y_0)\rangle$$ 
hence we further obtain 
\begin{align*}
	(\ell_1,\ell_2,\ell_3)= & \ell_1\cdot  (\ell_2\cdot\ell_3)-(\ell_1\cdot\ell_2)\cdot\ell_3 \\
	=&\langle\ell_1,y_0\rangle \langle\ell_2,y_0\rangle (E_0)^2(\ell_3) 
	-\langle\ell_1,y_0\rangle\langle\ell_2,
	(D_0)^2(y_0)\rangle\ell_3.
\end{align*}
Thus, by the symmetry hypothesis \eqref{LSA_prop_new_eq1}, we obtain that the associator  
$(\ell_1,\ell_2,\ell_3)$ 
is symmetric in $\ell_1$ and $\ell_2$.

	Finally, we check compatibility with the anchor $\rho^L$. 
	Since $\rho^L(\ell)=\langle \ell,y_0\rangle X_0\in T_R$ 
	and $\ell_1\cdot \ell_2=\langle \ell_1,y_0\rangle  E_0(\ell_2)-\langle \ell_1,D_0(y_0)\rangle\ell_2\in L $, we obtain 
\begin{align*}
	\rho^L(\ell_1\cdot \ell_2-\ell_2\cdot \ell_1)
	=
	&\langle \ell_1\cdot \ell_2-\ell_2\cdot \ell_1,y_0\rangle
	X_0\\
	=
	&\bigl((\langle \ell_1,y_0\rangle \langle E_0(\ell_2),y_0\rangle-
	\langle \ell_2,y_0\rangle\langle E_0(\ell_1),y_0\rangle)\\
	&-(\langle \ell_1,D_0(y_0)\rangle\langle \ell_2,y_0\rangle-
	\langle \ell_2,D_0(y_0)\rangle\langle \ell_1,y_0\rangle)\bigr)
	X_0.
\end{align*}
On the other hand, 
\begin{align*}
\rho^L(\ell_1)\rho^L(\ell_2)
=
&\langle \ell_1,y_0\rangle X_0(\langle \ell_2,y_0\rangle X_0) \\
=
&\langle \ell_1,y_0\rangle X_0(\langle \ell_2,y_0\rangle) X_0 
+\langle \ell_1,y_0\rangle \langle\ell_2,y_0\rangle (X_0)^2 \\
\mathop{=}\limits^{\eqref{LSA_prop_new_eq2}}
&
\langle \ell_1,y_0\rangle (\langle E_0(\ell_2),y_0\rangle+\langle \ell_2,D_0(y_0)\rangle)X_0
+\langle \ell_1,y_0\rangle \langle\ell_2,y_0\rangle (X_0)^2
\end{align*}
hence, comparing this equality with the above one, we see that 
$$	 \rho^L([\ell_1, \ell_2])=
\rho^L(\ell_1\cdot \ell_2-\ell_2\cdot \ell_1)=\rho^L(\ell_1)\rho^L(\ell_2)-\rho^L(\ell_2)\rho^L(\ell_1)
 =[\rho^L(\ell_1), \rho^L(\ell_2)]
.$$
This completes the proof. 
\end{proof}

\begin{remark}
\normalfont 
	In connection with the proof of Theorem~\ref{LSA_prop_new}, 
		we note that the compatibility of the anchor $\rho$ with the bracket $[\cdot,\cdot]$ on $L$ does not need the symmetry hypothesis~\eqref{LSA_prop_new_eq1}.
\end{remark}

We now apply Theorem~\ref{LSA_prop_new} for the $R$-bilinear duality map $\langle\cdot,\cdot\rangle\colon \Omega^1_R\times T_R\to R$. 

\begin{corollary}
\label{Bucharest}
For the $R$-module $\Omega^1_R$ and $X,Y\in T_R$, 
we define the anchor map 
\begin{equation}
	\label{Bucharest_eq1}
	\rho\colon \Omega^1_R\to T_R, \quad \rho(\alpha):=\langle\alpha,Y\rangle X.
\end{equation}
If the $R$-bilinear mapping 
\begin{equation}
	\label{Bucharest_eq2}
\Omega^1_R\times \Omega^1_R\to R,\quad (\alpha_1,\alpha_2)\mapsto \langle\alpha_1,Y\rangle\langle\alpha_2,[X,[X,Y]]\rangle
\end{equation}
is symmetric, then the mapping 
$$\Omega^1_R\times \Omega^1_R\to \Omega^1_R,\quad (\alpha_1,\alpha_2)\mapsto \alpha_1\cdot\alpha_2:=\nabla_{\alpha_1}\alpha_2:=\langle\alpha_1,Y\rangle\LD_X\alpha_2 -\langle\alpha_1,[X,Y]\rangle\alpha_2$$
defines the structure of a pre-Lie-Rinehart algebra on $\Omega^1_R$. 
\end{corollary}

\begin{proof}
As in the proof of Corollary~\ref{first}, we consider the $R$-modules $L:=\Omega^1_R$ 
and $N:=T_R$. 
Then equation~\eqref{LSA_prop_new_eq2} is satisfied for the Lie derivative $E_0:=\LD_X\colon \Omega^1_R\to \Omega^1_R$ 
hence the assertion follows by Theorem~\ref{LSA_prop_new}, 
applied for $D_0:=\ad_{T_R}X=[X,\cdot]$ and $y_0:=Y$. 
\end{proof}

\begin{corollary}
	\label{Bucharest_cor1}
Under the hypothesis of Corollary~\ref{Bucharest}, 
if we define 
\begin{equation}\label{Bucharest_cor1_eq1}
[\alpha_1,\alpha_2]:=\langle\alpha_1,Y\rangle\LD_X\alpha_2-\langle\alpha_2,Y\rangle\LD_X\alpha_1
-\langle\alpha_1,[X,Y]\rangle\alpha_2+\langle\alpha_2,[X,Y]\rangle\alpha_1
\end{equation}
for all $\alpha_1,\alpha_2\in \Omega^1_R$, 
then we obtain a Lie-Rinehart algebra structure on $\Omega^1_R$ with the anchor $\rho$ given by \eqref{Bucharest_eq1}. 
\end{corollary}

\begin{proof}
It suffices to note that, with the notation of Corollary~\ref{Bucharest}, 
we have $ [\alpha_1,\alpha_2]=\nabla_{\alpha_1}\alpha_2-\nabla_{\alpha_2}\alpha_1$. 
\end{proof}

\begin{remark}
\normalfont
Let us further note that the form of bracket~\eqref{Bucharest_cor1_eq1} consists of two fragments: 
\begin{align*}
& [\alpha_1,\alpha_2]_1=\langle\alpha_1,Y\rangle\LD_X\alpha_2
-\langle\alpha_2,Y\rangle\LD_X\alpha_1,\\
& [\alpha_1,\alpha_2]_2=-\langle\alpha_1,[X,Y]\rangle\alpha_2+\langle\alpha_2,[X,Y]\rangle\alpha_1,
\end{align*} 
each of which resembles the Lie bracket of the Lie algebra of a generalized $ax+b$-group, see \cite{BB18,DJ23}.
\end{remark}

\begin{corollary}
	\label{Bucharest_cor2}
If $X,Y\in T_R$ and there exists $c\in R$ satisfying 
\begin{equation}
		\label{Bucharest_cor2_eq1}
[X,[X,Y]]=cY
\end{equation}
then the bracket \eqref{Bucharest_cor1_eq1} gives a Lie-Rinehart algebra structure on $\Omega^1_R$ with the anchor~$\rho$ given by~\eqref{Bucharest_eq1}. 
\end{corollary}

\begin{proof}
If \eqref{Bucharest_cor2_eq1} is satisfied, then it is easily seen that the mapping \eqref{Bucharest_eq2} is symmetric, hence the assertion follows by Corollary~\ref{Bucharest_cor1}. 
\end{proof}

\begin{remark}
\normalfont 
Let us further note that the multiplication \eqref{multiplication-1} under appropriate assumptions defines the a left-symmetric algebra on $L$, i.e., an algebra whose associator is symmetric in the first two arguments 
$$
(\ell_1,\ell_2,\ell_3)=(\ell_2,\ell_1,\ell_3)
$$
for any $\ell_1,\ell_2,\ell_3\in L$.
The same is true for the multiplication \eqref{multiplication-2}.
Specific examples of such algebras are given in Corollary~\ref{first} under condition \eqref{first0_eq1}, as well as in Corollary~\ref{Bucharest} under condition \eqref{Bucharest_cor2_eq1}. Such algebras have applications in integrable systems theory and are related to the notion of the classical $r$-matrix, see \cite{Bai04}.
\end{remark}

\begin{example}
\label{A_ex}
\normalfont
In the special case $\KK=\RR$ and $R=\Ci(\RR)$, the above equations \eqref{first0_eq1}~and~\eqref{Bucharest_cor2_eq1} can be expressed in terms of ordinary differential equations. 
In fact, for arbitrary $g\in \Ci(\RR)$, let us consider the vector field $D_g:=g\frac{\de}{\de t}\in\Der(\Ci(\RR))$, 
where we denote by $t$ the coordinate in~$\RR$. 
Then for all $f,g\in \Ci(\RR)$ we have 
\begin{equation}
\label{A_ex_eq1}
[D_f,D_g]=D_{fg'-f'g}.
\end{equation}
If moreover $g\in\Ci(\RR)$ satisfies $g(t)\ne0$ for every $t\in\RR$, then $1/g\in\Ci(\RR)$ hence 
for every $f\in\Ci(\RR)$ we have  
\begin{equation}
\label{A_ex_eq2}
D_{fg'-f'g}=(fg'/g-f')D_g. 
\end{equation}
For $X:=D_1=\frac{\de}{\de t}$, $Y:=D_g$ with $g\in\Ci(\RR)$, and $c\in\Ci(\RR)$,
the equation~\eqref{Bucharest_cor2_eq1} is equivalent via \eqref{A_ex_eq1} to the ordinary differential equation
\begin{equation}
	\label{A_ex_eq3}
g''(t)=c(t)g(t)\text{ for all }t\in\RR.
\end{equation}
Thus, in the special case $c(t)=t$ for all $t\in\RR$ we obtain Airy's differential equation (e.g., \cite[\S 11.1]{Ol97}), 
and it follows that the Airy functions give rise to new Lie algebroid structures on the cotangent bundle of the real line~$\RR$, using Corollary~\ref{Bucharest_cor2}.

Similarly, the equation~\eqref{first0_eq1} is equivalent to the ordinary differential equation 
\begin{equation}
	\label{A_ex_eq4}
	g'(t)=c(t)g(t)\text{ for all }t\in\RR.
\end{equation}
It follows by the uniqueness of the solutions to Cauchy problems for first order linear ordinary differential equations (e.g., \cite[Prop. 1.2.4]{Ho03}) that if $g\in\Ci(\RR)$ satisfies \eqref{A_ex_eq4} and there exists $t_0\in\RR$ with $g(t_0)=0$, then $g(t)=0$ for every $t\in\RR$. 
On the other hand, it is well known that the Airy functions ${\rm Ai}$ and ${\rm Bi}$ 
do have zeros, namely in the interval $(-\infty,0)$. 
(See e.g., their graphs in \cite[\S 11.1.3]{Ol97}.) 
Therefore, if $g$ is any of the functions  ${\rm Ai}$ and ${\rm Bi}$, then there is no function $c\in \Ci(\RR)$ satisfying \eqref{A_ex_eq4}, or, equivalently, $[X,Y]=cY$ for 
$X=D_1=\frac{\de}{\de t}$ and $Y=D_g$. 
However,  as noted above, we have \eqref{A_ex_eq3} for $c(t)=t$ for all $t\in\RR$, 
hence $[X,[X,Y]]=cY$ for 
$X=D_1=\frac{\de}{\de t}$ and $Y=D_g$. 
Thus, in this concrete situation, Corollary~\ref{first0} is not applicable, while Corollary~\ref{Bucharest_cor2} is.

For the $R$-module $\Omega^1_R$ the Lie bracket \eqref{Bucharest_cor1_eq1} has the form
\begin{equation}
[f\de t, h\de t]=-g(t)\mathcal{D}_{H}(f,h)(t)\de t \nonumber
\end{equation}
for all $f,h\in  \Ci(\RR)$, where $\Dc_{H}$ denotes Hirota's operator (see \cite{To89}) which acts on the pair of functions $f$ and $h$ in the following way
\begin{equation}
\mathcal{D}_{H}(f,h)(t):=\left( \frac{\de}{\de t} - \frac{\de}{\de \tilde{t}} \right) f(t)h(\tilde{t})\bigg|_{t=\tilde{t}}=f'(t)h(t)-f(t)h'(t). \nonumber
\end{equation}
This operator is used in the method of finding soliton solutions for non-linear equations, as for the example KdV.

\end{example}

\begin{remark}
\normalfont
	Lie-Rinehart algebra structures on the space of differential forms $\Omega^1_R$ have been constructed before in the algebraic theory of Dirac structures, motivated by the study of integrability of certain nonlinear differential equations. 
		See \cite[Ex. 2.3 and Th. 2.12]{Do93}, where the Hamiltonian operator $H$ plays the role of the anchor. 		
		From this perspective, we note that the Lie-Rinehart brackets constructed in Corollaries \ref{first} and \ref{Bucharest_cor1} above do not seem to be directly obtainable from the formalism of Dirac structures. 
		For instance, in the setting of Corollaries \ref{first}~and ~\ref{Bucharest}, the graph of the anchor map 
		$$H:=\rho\colon\Omega^1_R\to T_R, \quad H(\alpha):=\langle \alpha,Y\rangle X,$$ 
 is 
 $$\Gamma:=\{(\langle \alpha,Y\rangle X, \alpha)| \alpha\in \Omega^1_R\}\subseteq T_R\times \Omega^1_R$$
 whose orthogonal complement with respect to the bilinear map 
 $$\langle\cdot,\cdot\rangle\colon (T_R\times \Omega^1_R)\times (T_R\times \Omega^1_R)\to R, \quad 
 \langle(Z_1,\alpha_1),(Z_2,\alpha_2)\rangle:=\langle\alpha_1,Z_2\rangle+\langle\alpha_2,Z_1\rangle$$
 satisfies  
 $\Gamma^{\perp}\supseteq\{(-\langle \alpha,X\rangle Y, \alpha)| \alpha\in \Omega^1_R\}$. 
 (Here we actually have equality if the pairing $\Omega^1_R\times T_R\to R$ is nondegenerate.)
 Thus, in general, we have $\Gamma\ne\Gamma^\perp$, hence $\Gamma$ is not a Dirac structure 
 in the sense of \cite[\S 2.3]{Do93}, which in turn shows that the above anchor $\rho$ is not a Hamiltonian operator in the sense of \cite[\S 2.6]{Do93}.
 This is also reflected by the fact that 
 the expression  $\langle\rho(\de r_1),\de r_2\rangle =Y(r_1)X(r_2)$, 
	is not skew symmetric with respect to $r_1,r_2\in R$, unlike the Poisson bracket in \cite[Eq. (2.21)]{Do93}.  
\end{remark}

\subsection*{Acknowledgment}

The first-named author D.B. acknowledges partial financial support from 
the Research Grant GAR 2023 (code 114), supported from the Donors' Recurrent Fund of the Romanian Academy, managed by the "PATRIMONIU" Foundation. 
The third-named author G.J. was partially supported by National Science Centre, Poland project 2020/01/Y/ST1/00123.

\end{document}